\documentclass[12pt,reqno]{amsart}
\usepackage[skip=5pt plus1pt, indent=20pt]{parskip}

\usepackage{palatino}
\usepackage{mathpazo}
\usepackage{amsmath,amssymb,amsxtra,color,calligra,mathrsfs,tcolorbox,stmaryrd}

\usepackage{needspace}

\usepackage{xcolor} 
\colorlet{mdtRed}{red!50!black}
\colorlet{dblue}{blue!50!black}
\usepackage[colorlinks,pagebackref=true]{hyperref}
\hypersetup{linkcolor=dblue,citecolor=dblue,filecolor=dullmagenta,urlcolor=mdtRed}
\renewcommand*{\backref}[1]{}
\renewcommand*{\backrefalt}[4]{[{%
		\ifcase #1 Not cited.%
		\or $\uparrow$~#2.%
		\else $\uparrow$~#2.%
		\fi%
	}]}
\usepackage[all]{xy}
\usepackage{tikz,tikz-cd,tkz-graph,enumerate}
\usetikzlibrary{matrix,arrows,decorations.pathmorphing}

\usepackage[us,12hr]{datetime}
\usepackage{comment}

\newcommand{\mf}[1]{\mathfrak{#1}}
\newcommand{\mc}[1]{\mathcal{#1}}

\newcommand{\bb}[1]{\mathbb{#1}}

\renewcommand{\ker}{\mathrm{Ker}}

\newcommand{\doi}[1]{\href{https://doi.org/#1}{doi:#1}}

\numberwithin{equation}{subsection}

\newtheorem{theorem}[equation]{Theorem}

\newtheorem{lemma}[equation]{Lemma}
\newtheorem{proposition}[equation]{Proposition}

\theoremstyle{definition}
\newtheorem{definition}[equation]{Definition}
\newtheorem{remark}[equation]{Remark}
\newtheorem{example}[equation]{Example}

\DeclareMathOperator{\Id}{{\rm Id}}

\DeclareMathOperator{\At}{At}
\DeclareMathOperator{\ad}{ad}

\DeclareMathOperator{\End}{End}
\DeclareMathOperator{\Hom}{Hom}

\DeclareMathOperator{\Lie}{Lie}

\DeclareMathOperator{\GL}{GL}

\title[Lie Algebroid Connections on Principal Bundles]{Lie Algebroid Connections on Principal Bundles}

\author{Samit Ghosh}

\address{Department of Mathematics and Statistics, Indian Institute of Science Education and Research Kolkata, Mohanpur 741246, Nadia, West Bengal, India}

\email{sg23rs005@iiserkol.ac.in}

\author{Arjun Paul}

\address{Department of Mathematics and Statistics, Indian Institute of Science Education and Research Kolkata, Mohanpur 741246, Nadia, West Bengal, India}

\email{arjun.paul@iiserkol.ac.in}

\thanks{Proc. Indian Acad. Sci. (Math. Sci.) vol. 136, no. 36 (2026); \doi{10.1007/s12044-026-00884-3}}

\thanks{Corresponding author: Arjun Paul}



\date{30 September 2026}

\begin{document}
	
	\maketitle
	
	\begin{abstract}
		Let $X$ be an irreducible smooth complex projective variety. Let $G$ be a linear algebraic group over $\mathbb{C}$. We define the notion of Lie algebroid valued connection on holomorphic principal $G$--bundles on $X$, and study their basic properties under extension and reduction of structure group. Finally we address the question of existence of a Lie algebroid connection on principal $G$--bundles on smooth complex projective curves. 
	\end{abstract}
	
	\tableofcontents

	\section{Introduction}\label{sec:introduction}
	
	A famous theorem of A. Weil \cite{Weil-1938} says that a holomorphic vector bundle $E$ 
	on a compact connected Riemann surface $X$ admits a holomorphic connection if and only if 
	each indecomposible holomorphic direct summand of $E$ has degree zero. 
	In \cite{Atiyah-1957} M. Atiyah generalizes the notion of holomorphic connections 
	in the context of holomorphic principal $G$--bundles on compact K\"ahler manifolds, 
	and gives an algebro-geometric proof of Weils' theorem for holomorphic vector bundles 
	on compact connected Riemann surfaces. 
	In \cite{Azad-Biswas-2002} Azad and Biswas generalize Weils theorem for 
	holomorphic principal $G$--bundles on compact connected Riemann surfaces. 
	It is clear from these results that not every holomorphic vector bundle 
	and principal $G$--bundle can admit holomorphic connections. 
	This naturally leads one to consider the notion of meromorphic connections. 
	One of the simplest kind of meromorphic connections is the notion of 
	logarithmic connection, which are treated for holomorphic vector bundles 
	and holomorphic principal $G$--bundles, for example in 
	\cite{Biswas-Dan-Paul-2018}, \cite{Biswas-Dan-Paul-Saha-2017}, \cite{Gurjar-Paul-2020} etc. 
	
	In the context of complex algebraic and differential geometry, the classical notion 
	of holomorphic as-well-as singular connections has natural generalization by 
	replacing tangent bundle with a Lie algebroid leading to the notion of 
	{\it Lie algebroid connections}, which is more convenient to work with in some setups 
	like Poisson geometry, foliation theory etc. The notion of Lie algebroid connections 
	also generalize the notion of holomorphic and logarithmic connections. 
	It is an interesting problem to study Lie algebroid connections on holomorphic 
	vector bundles and principal bundles. 
	
	Let $X$ be a connected compact Riemann surface. Fix a holomorphic Lie algebroid $\mc V = (V, [\cdot\,,\,\cdot], \varphi)$ on $X$ with $V$ a stable vector bundle. In \cite{Biswas-Kumar-Singh-2024} the authors shows that every holomorphic vector bundle on $X$ admits a $\mc V$--valued Lie algebroid connection generalizing a result \cite[Corollary 3.17]{Alfaya-Oliveire-2024}. It is natural to generalize the notion of Lie algebroid connection on principal $G$--bundles and study the question of existence of Lie algebroid connection on holomorphic or algebraic principal $G$--bundles. 
	
	In this paper we generalize the notion of $\mc V$--valued Lie algebroid connections in the context of principal $G$--bundles in \S\,\ref{sec:sec2}. In \S\,\ref{sec:basic-properties} we study Lie algebroid connections under extension and reduction of the structure group of the principal bundles. Finally in \S\,\ref{sec:Existence} we address the question of existence of a Lie algebroid connection on principal $G$--bundles over a smooth complex projective curve $X$. In \S\,\ref{sec:semistability} we briefly recall the notion of semistable principal $G$--bundles, and then in \S\,\ref{sec:semistable-case-existence} we prove the following result which generalize \cite[Corollary 3.17]{Alfaya-Oliveire-2024} in the setup of principal $G$--bundles.  
	
	\begin{proposition}
		Let $X$ be an irreducible smooth complex projective curve. Fix a Lie algebroid $\mc V = (V, [\cdot\,,\,\cdot], \varphi)$ on $X$ such that $V$ is a stable vector bundle on $X$ with the slope $\mu(V) \neq 2-2g$, where $g$ is the genus of $X$. Let $G$ be a reductive affine algebraic group over $\bb C$. Then any semistable principal $G$--bundle $E_G$ on $X$ admits a $\mc V$--valued Lie algebroid connection. 
	\end{proposition}
	
	The case when $E_G$ is not semistable is addressed in \S\,\ref{sec:main-thm}. Here we consider the Harder-Narasimhan filtration of the adjoint bundle $\ad(E_G)$ of $E_G$, and use the Harder-Narasimhan reduction of $E_G$ to reduce to the case of semistable principal bundle. Then we use Proposition \ref{prop:existence_semistable-G-bundles-on-curve} and Proposition \ref{prop:connection-on-E_P-from-E_L} to prove the following general case. 
	
	\begin{theorem}
		Let $X$ be an irreducible smooth complex projective curve of genus $g \geq 2$. Fix a Lie algebroid $\mc V = (V, [\cdot\,,\,\cdot], \varphi)$ on $X$ such that $V$ is a stable vector bundle on $X$ with the slope $\mu(V) \neq 2-2g$. Let $G$ be a reductive linear algebraic group over $\bb C$. Then any holomorphic principal $G$--bundle $E_G$ on $X$ admits a $\mc{V}$--valued Lie algebroid connection. 
	\end{theorem}
	
	This generalize the main result of \cite{Biswas-Kumar-Singh-2024} to the case of holomorphic principal $G$--bundles on $X$.

	\section{Lie Algebroid Connections}\label{sec:sec2}
	
	\subsection{The case of vector bundles}
	Let $X$ be an irreducible smooth projective variety over $\mathbb{C}$. 
	Let $\mc O_X$ be the sheaf of holomorphic functions on $X$, and let 
	$TX$ be the holomorphic tangent bundle of $X$. 
	
	\begin{definition}\cite[\S\,1.1]{Alfaya-Oliveire-2024}
		A {\it Lie algebroid} on $X$ is a triple $\mc V := (V, [\cdot, \cdot], \varphi)$, where 
		\begin{enumerate}[(i)]
			\item $V$ is a holomorphic vector bundle on $X$, 
			
			\item $[\cdot, \cdot] : V\times V \to V$ is a $\bb C$--bilinear skew-symmetric 
			morphism of sheaves such that for all locally defined sections 
			$u, v, w$ of $V$, the following Jacobi identity holds: 
			$$[u, [v, w]]+[v, [w, u]]+[w, [u, v]] = 0;$$
			
			\item $\varphi : V \to TX$ is a vector bundle homomorphism 
			satisfying the following properties: for all locally defined sections $s, t$ of $V$ and 
			locally defined section $f$ of $\mc O_X$, we have 
			\begin{enumerate}[(a)]
				\item {\it Compatibility of Lie algebra structures}: 
				$\varphi([s, t]) = [\varphi(s), \varphi(t)]$, and 
				\item {\it Leibniz rule}: $[fs, t] = f[s, t] - \varphi(t)(f)s$.
			\end{enumerate}
		\end{enumerate}
		The homomorphism $\varphi$ is called the {\it anchor map} of the Lie algebroid $\mc V$. 
		The {\it degree} and the {\it rank} of $\mc V$ is defined to be the degree and the rank, 
		respectively, of the underlying vector bundle $V$ of $\mc V$. 
	\end{definition}
	
	The dual of the anchor map gives a holomorphic vector bundle homomorphism 
	$$\varphi^* : \Omega_X^1 \longrightarrow V^*,$$
	where $\Omega_X^1$ is the holomorphic cotangent bundle of $X$. 
	Fix a Lie algebroid $\mc V := (V, [\cdot, \cdot], \varphi)$ on $X$. 
	Let $\mc E$ be a holomorphic vector bundle on $X$. 
	
	\begin{definition}\label{def:Lie-alg-conn-on-VB}
		A {\it $\mc V$--valued Lie algebroid connection} on  
		$\mc E$ on $X$ is a $\bb C$--linear homomorphism of sheaves 
		$$D : \mc E \longrightarrow \mc E\otimes V^*$$ 
		satisfying the {\it $\varphi^*$--twisted Leibniz rule}: 
		\begin{equation}\label{eqn:phi^*-twisted-Leibniz-rule}
			D(f\cdot s) = fD(s)+s\otimes\varphi^*(df), 
		\end{equation}
		for all locally defined section $s$ of $\mc E$ and for all locally defined 
		section $f$ of $\mc O_X$. 
	\end{definition}

	\subsection{The case of principal $G$--bundles}
	Let $G$ be a linear algebraic group over $\mathbb{C}$ 
	with the Lie algebra $\mf{g} := \Lie(G)$. 
	Let $p : E_G \to X$ be a holomorphic (resp., algebraic) principal $G$--bundle on $X$. 
	The adjoint representation 
	$${\rm ad} : G \longrightarrow \GL(\mf{g})$$ 
	of $G$ on its Lie algebra $\mf{g}$ gives rise to a vector bundle 
	$$\ad(E_G) := E_G\times^{\rm ad}\mf{g}$$ 
	on $X$, called the {\it adjoint vector bundle} of $E_G$. 
	If $E$ is the frame bundle of a vector bundle $\mc E$ of rank $n$ on $X$, 
	then we have $\ad(E) \cong \End(\mc E)$, the endomorphism bundle of $\mc E$. 
	The surjective submersion $p : E_G \to X$ gives rise to an exact sequence of 
	vector bundles 
	\begin{equation}\label{eqn:Atiyah-exact-seq-holomorphic}
		\xymatrix{
			0 \ar[r] & \ad(E_G) \ar[r] & \At(E_G) \ar[r]^-{d'p} & TX \ar[r] & 0 
		}
	\end{equation}
	called the Atiyah exact sequence of $E_G$. 
	Recall that a {\it holomorphic} (resp., {\it algebraic}) connection on $E_G$ 
	is an $\mc O_X$--module homomorphism $\nabla : TX \to \At(E_G)$ 
	such that $d'p\circ\nabla = \Id_{TX}$; see \cite{Atiyah-1957}. 
	
	As before, fix a holomorphic (resp., algebraic) Lie algebroid 
	$\mc V = (V, [\cdot\,,\,\cdot], \varphi)$ on $X$. 
	To extend the notion of $\mc V$--valued Lie algebroid connection to the 
	case of holomorphic (resp., algebraic) principal $G$--bundle $E_G$ on $X$, 
	we need to modify the Atiyah bundle $\At(E_G)$ in such a way which essentially 
	keep track of the sections of $\At(E_G)$ whose image under the map $d'p$ in 
	\eqref{eqn:Atiyah-exact-seq-holomorphic} lands inside the image of 
	the anchor map $\varphi : V \to TX$ in an appropriate way. 
	For this purpose following \cite{Biswas-Kannan-Nagaraj-2015, Biswas-Paul-2017} we consider the map 
	$$\rho : \At(E_G)\oplus V \longrightarrow TX$$ 
	defined by 
	\begin{equation}
		\rho(\xi, v) = d'p(\xi)-\varphi(v), 
	\end{equation}
	for all locally defined section $\xi$ of $\At(E_G)$ and locally defined section 
	$v$ of $V$. Note that $\rho$ is a surjective homomorphism of vector bundles, and 
	\begin{equation}
		\At_{\varphi}(E_G) := \rho^{-1}(0) 
	\end{equation}
	is a vector bundle on $X$. 
	Note that $\At_{\varphi}(E_G) \cong \At(E_G)\times_{d'p,\,TX,\,\varphi}\,V$. 
	The restriction of the second projection map gives rise 
	to a vector bundle homomorphism 
	\begin{equation}\label{eqn:rho-tilde-map}
		\widetilde{\rho} : \At_{\varphi}(E_G) \longrightarrow V
	\end{equation} 
	with kernel 
	$$\ker(\widetilde{\rho}) = \ad(E_G).$$ 
	Thus we have the following short exact sequence 
	\begin{equation}\label{eqn:Atiyah-exact-seq-for-(V,phi)-valued-connection}
		0 \longrightarrow \ad(E_G) \longrightarrow \At_{\varphi}(E_G) 
		\stackrel{\widetilde{\rho}}{\longrightarrow} V \longrightarrow 0 
	\end{equation}
	of vector bundles on $X$, which fits into the following 
	commutative diagram 
	\begin{equation}\label{eqn:diagram-for-defn-of-Lie-alg-conn}
		\begin{gathered}
			\xymatrix{
				0 \ar[r] & \ad(E_G) \ar[r] \ar@{=}[d] & \At_{\varphi}(E_G) \ar[r]^-{\widetilde{\rho}} \ar[d] 
				& V \ar[r] \ar[d]^-\varphi & 0 \\ 
				0 \ar[r] & \ad(E_G) \ar[r] & \At(E_G) \ar[r]^-{d'p} & TX \ar[r] & 0 
			}
		\end{gathered}
	\end{equation}
	of vector bundle homomorphisms with all rows exact. 
	
	\begin{definition}\label{def:Lie-alg-conn-on-G-Bundles}
		A {\it $\mc V$--valued Lie algebroid connection} on $E_G$ is a 
		vector bundle homomorphism 
		$$\nabla : V \longrightarrow \At_{\varphi}(E_G)$$
		such that $\widetilde{\rho}\circ\nabla = \Id_V$, 
		where $\widetilde{\rho}$ is defined in \eqref{eqn:rho-tilde-map}. \,\,
	\end{definition}
	

	\begin{example}
		\mbox{} 
		\begin{enumerate}[(i)]
			\item Let $\varphi : V \to TX$ be an isomorphism of Lie algebroids 
			(i.e., isomorphism of $\mc O_X$--modules preserving the Lie bracket). 
			Then it follows from the above commutative diagram 
			\eqref{eqn:diagram-for-defn-of-Lie-alg-conn} that giving a $\mc V$--valued 
			Lie algebroid connection on $E_G$ is the same as giving a holomorphic 
			connection on $E_G$. 
			
			\item Let $D$ be a reduced effective divisor on $X$, let $V = TX(-\log D)$ be the 
			logarithmic tangent bundle of $X$, and let $\varphi : TX(-\log D) \hookrightarrow TX$ 
			be the inclusion map. Then it follows from the commutative diagram 
			\eqref{eqn:diagram-for-defn-of-Lie-alg-conn} that giving a $\mc V$--valued Lie algebroid 
			connection on $E_G$ is the same as giving a logarithmic connection on $E_G$ singular 
			along the divisor $D$. 
		\end{enumerate} 
	\end{example}

	\begin{remark}
		Fix a holomorphic Lie algebroid $\mathcal{V} = (V, [\cdot\,,\,\cdot], \varphi)$ on $X$. Let $E$ be a holomorphic vector bundle of rank $n$ on $X$, and let $\widetilde{E}$ be the holomorphic frame bundle of $E$. Note that $\widetilde{E}$ is a holomorphic principal $GL_n(\mathbb{C})$--bundle on $X$. 
		Then $E$ admits a $\mathcal{V}$--valued Lie algebroid connection if and only if $\widetilde{E}$ admits a $\mathcal{V}$--valued Lie algebroid connection. 
		Indeed, given a $\mathcal{V}$--valued Lie algebroid connection $D : E \to E\otimes V^*$ on $E$, for each local section $v$ of $V$ contraction with $D$ defines a first order differential operator 
		$$D_v : E \to E, \ \ s \mapsto D_v(s) = \langle\,D(s), v\rangle $$ 
		on $E$ with the principal symbol $\varphi(v) \in TX$, which by \cite{Atiyah-1957} corresponds to a section $\zeta_v$ of $\At(\widetilde{E})$ satisfying $d'p(\zeta_v) = \varphi(v)$. Consequently the pair $(\zeta_v, v)$ lies in  $\At_\varphi(\widetilde{E}) = \rho^{-1}(0)$, and hence the assignment $v \mapsto (\zeta_v, v)$ defines a vector bundle homomorphism $\nabla : V \to \At_{\varphi}(\widetilde{E})$ satisfying $\widetilde{\rho}\circ\nabla = \Id_V$. 
		Conversely, suppose that $\widetilde{E}$ admits a $\mc V$--valued Lie algebroid connection $$\nabla : V \to \At_{\varphi}(\widetilde{E}).$$ Since $\widetilde{\rho}\circ\nabla = \Id_V$, given any local section $v$ of $V$ its image $\nabla(v)$ is given by a pair $(\zeta_v, v) \in \At(\widetilde{E})\oplus V$ satisfying $\rho(\zeta_v, v) = 0$, which gives $d'p(\zeta_v) = \varphi(v)$. Then $\zeta_v \in \At(\widetilde{E})$ corresponds to a first order differential operator on $E$ with the principal symbol $\varphi(v)$. Then for given a local section $f$ of $\mc O_X$ and a local section $s$ of $E$, we have 
		$$\zeta_v(f\cdot s) = f\zeta_v(s)+\varphi(v)(f)s.$$
		This gives a $\bb C$--linear map $D : E \to E\otimes V^*$ satisfying the relation 
		$$\langle D(s), v\rangle = \zeta_v(s), \ \forall\ v \in V.$$ 
		Then we have 
		\begin{align*}
			\langle D(f\cdot s), v\rangle = \zeta_v(f\cdot s) 
			= f\zeta_v(s)+\varphi(v)(f)s 
			= f\langle D(s), v\rangle + \langle s\otimes \varphi^*(df), v\rangle\,. 
		\end{align*}
		Since the above identity holds for all local sections $v$ of $V$, it follows that 
		$D(f\cdot s) = fD(s)+s\otimes\varphi^*(df)$. Thus we get a $\mc V$--valued Lie algebroid connection on $E$. 
	\end{remark}

	The short exact sequence \eqref{eqn:Atiyah-exact-seq-for-(V,phi)-valued-connection} 
	defines a cohomology class 
	\begin{equation}\label{defn:V-valued-Atiyah-class-of-E_G}
		\Phi_{\mc V}(E_G) \in H^1(X, \ad(E_G)\otimes V^*), 
	\end{equation}
	such that the exact sequence \eqref{eqn:Atiyah-exact-seq-for-(V,phi)-valued-connection} 
	splits holomorphically if and only if $\Phi_{\mc V}(E_G) = 0$. 
	Thus we have the following result. 
	
	\begin{proposition}\label{prop:E_G-admits-Lie-alg-conn-iff-Atiyah-cls-vanishes}
		A holomorphic principal $G$--bundle $E_G$ on $X$ admits a $\mc V$--valued 
		holomorphic Lie algebroid connection if and only if $\Phi_{\mc V}(E_G) = 0$. 
		We call $\Phi_{\mc V}(E_G)$ the {\it $\mc V$--valued Atiyah class of $E_G$}. 
	\end{proposition}

	Let $\nabla : V \to \At_{\varphi}(E_G)$ be a $\mc V$--valued Lie algebroid connection 
	on $E_G$ over $X$. For all locally defined holomorphic sections $s$ and $t$ of $V$, let 
	$$\kappa_{\nabla}(s, t) := [\nabla(s), \nabla(t)] - \nabla([s, t]).$$ 
	Since the homomorphism $\widetilde{\rho} : \At_{\varphi}(E_G) \to V$ respects the Lie algebra 
	structures on the sheaves of sections, $\kappa_{\nabla}(s, t)$ defines a holomorphic local 
	section of $\ad(E_G)$. Thus we obtain a section 
	$$\kappa_{\nabla} \in H^0(X, \ad(E_G)\otimes \bigwedge\nolimits^2V^*),$$
	called the {\it curvature} of the $\mc V$--valued Lie algebroid connection 
	$\nabla$ on $E_G$. The section $\kappa_{\nabla}$ can be considered as an obstruction 
	for $\nabla$ to be a Lie algebra homomorphism. 
	
	\begin{definition}
		A $\mc V$--valued Lie algebroid connection $\nabla$ on a 
		principal $G$--bundle $E_G$ on $X$ is said to be {\it flat} if $\kappa_{\nabla} = 0$. 
	\end{definition}
	
	\begin{proposition}
		If ${\rm rank}(\mc V) = 1$, any $\mc V$--valued 
		Lie algebroid connection on $E_G$ is flat. 
	\end{proposition}
	
	\begin{proof}
		If ${\rm rank}(\mc V) = 1$, then $\bigwedge\nolimits^2 V^* = 0$ and so for any 
		$\mc V$--valued connection $\nabla$ on $E_G$, its curvature 
		$\kappa_{\nabla}$, being an element of $H^0(X, \ad(E_G)\otimes \bigwedge\nolimits^2V^*) = 0$, 
		vanishes identically. This completes the proof. 
	\end{proof}

	\section{Basic Properties}\label{sec:basic-properties}
	
	\subsection{Extension of structure groups}
	Let $G$ and $H$ be linear algebraic groups over $\bb C$ with their Lie algebras 
	$\mf g$ and $\mf h$, respectively. Given a homomorphism of algebraic groups $f : G \to H$ 
	over $\bb C$, let $df : \mf g \to \mf h$ be the Lie algebra homomorphism induced by $f$. 
	Let $p : E_G \to X$ be a holomorphic principal $G$--bundle over $X$, and let 
	$$p' : E_H := E_G\times^f H \to X$$ 
	be the associated principal $H$--bundle on $X$ obtained by extending the structure group 
	of $E_G$ along $f$. Let 
	\begin{align*}
		& \ad(f) : \ad(E_G) \longrightarrow \ad(E_H) \\ 
		\text{and} \ \ 
		& \At(f) : \At(E_G) \longrightarrow \At(E_H) 
	\end{align*}
	be the homomorphisms of the adjoint bundles and the Atiyah bundles of $E_G$ and $E_H$, 
	respectively, induced by $f$. Then we have the following commutative diagram of vector bundle 
	homomorphisms 
	\begin{equation}\label{diag:comm-diag-of-Atiyah-exact-seqn-for-extn-of-str-grp}
		\begin{gathered}
			\xymatrix{
				0 \ar[r] & \ad(E_G) \ar[d]^{\ad(f)} \ar[r]^{\iota_G} & \At(E_G) \ar[d]^{\At(f)} \ar[r]^-{d'p} 
				& TX \ar@{=}[d] \ar[r] & 0 \\ 
				0 \ar[r] & \ad(E_H) \ar[r]^{\iota_H} & \At(E_H) \ar[r]^-{d'p'} & TX \ar[r] & 0\,. 
			}
		\end{gathered}
	\end{equation}
	It is clear from the above diagram that a holomorphic connection on $E_G$ induces a 
	holomorphic connection on $E_H := E_G\times^fH$. 
	Let $$\rho' : \At(E_H)\oplus V \to TX$$ be the homomorphism defined by 
	$$\rho'(\xi, v) = d'p'(\xi) - \varphi(v),$$
	for all locally defined sections $\xi$ of $\At(E_H)$ and $v$ of $V$, respectively. 
	Let $$\widetilde{\rho'} : \At_{\varphi}(E_H) := \ker(\rho') \longrightarrow V$$ 
	be the restriction of the second projection map. 
	Then we have a vector bundle homomorphism 
	$$\At_{\varphi}(f) : \At_{\varphi}(E_G) \longrightarrow \At_{\varphi}(E_H)$$ 
	such that $\widetilde{\rho'}\circ\At_{\varphi}(f) = \widetilde{\rho}$. 
	Thus, the above commutative diagram 
	\eqref{diag:comm-diag-of-Atiyah-exact-seqn-for-extn-of-str-grp} induces the 
	following commutative diagram of vector bundles and homomorphisms 
	\begin{equation}\label{diag:phi-twisted-diag-of-Atiyah-exact-seqns-for-extn-of-str-grp}
		\begin{gathered}
			\xymatrix{
				0 \ar[r] & \ad(E_G) \ar[d]^{\ad(f)} \ar[r] & \At_{\varphi}(E_G) 
				\ar[d]^{\At_{\varphi}(f)} \ar[r]^-{\widetilde{\rho}} & V \ar@{=}[d] \ar[r] & 0 \\ 
				0 \ar[r] & \ad(E_H) \ar[r] & \At_{\varphi}(E_H) \ar[r]^-{\widetilde{\rho'}} & 
				V \ar[r] & 0\,. 
			}
		\end{gathered}
	\end{equation}
	From this, we have a natural homomorphism of cohomologies 
	\begin{equation}
		H^1(f) : H^1(X, \ad(E_G)\otimes V^*) \longrightarrow H^1(X, \ad(E_H)\otimes V^*) 
	\end{equation}
	such that $H^1(f)(\Phi_{\mc V}(E_G)) = \Phi_{\mc V}(E_H)$. 
	As an immediate consequence of it, we have the following result.

	\begin{proposition}\label{prop:Lie-algebroid-connection-under-extn-of-str-grp}
		Let $f : G \to H$ be a homomorphism of linear algebraic groups over $\mathbb{C}$. 
		Let $E_G$ be a holomorphic principal $G$--bundle on $X$, and let $E_H$ be the 
		holomorphic principal $H$--bundle on $X$ obtained from $E_G$ by extension of its 
		structure group along $f$. Then any $\mc V$--valued Lie algebroid connection on 
		$E_G$ induces a $\mc V$--valued Lie algebroid connection on $E_H$.
	\end{proposition}
	
	\begin{proof}
		If $E_G$ admits a $\mc V$--valued Lie algebroid connection, 
		then $\Phi_{\mc V}(E_G) = 0$. 
		Since $$\Phi_{\mc V}(E_H) = H^1(f)(\Phi_{\mc V}(E_G)) = 0,$$ 
		the result follows from 
		Proposition \ref{prop:E_G-admits-Lie-alg-conn-iff-Atiyah-cls-vanishes}
	\end{proof}

	\subsection{Reduction of structure group}
	
	Now it is interesting to ask the following question: 
	Suppose that $f : G \to H$ be a homomorphism of linear algebraic groups over $\bb C$. 
	If $E_H$ admits a $\mc V$--valued connection, does $E_G$ admits a $\mc V$--valued connection? 
	We give partial answers to this question. 
	
	\begin{proposition}\label{prop:Key-reduction-for-reductive-subgroup}
		Let $f : G \to H$ be an injective homomorphism of linear algebraic groups over $\bb C$ 
		with $G$ reductive. Let $E_G$ be a principal $G$--bundle on $X$, and let 
		$$E_H = E_G\times^f H$$ 
		be the principal $H$--bundle on $X$ obtained by extending the structure group of 
		$E_G$ along the homomorphism $f$. 
		If $E_H$ admits a $\mc V$--valued Lie algebroid connection, then $E_G$ admits a 
		$\mc V$--valued Lie algebroid connection. 
	\end{proposition}
	
	\begin{proof}
		Let $\mf g := {\rm Lie}(G)$ and $\mf h := {\rm Lie}(H)$ be the Lie algebras of $G$ 
		and $H$, respectively. Let $df : \mf g \to \mf h$ be the Lie algebra homomorphism 
		induced by $f$, and let 
		$$\ad(f) : \ad(E_G) \longrightarrow \ad(E_H)$$
		be the vector bundle homomorphism induced by $df$.  
		Let $\alpha : G \to \End(\mf g)$ and $\beta : H \to \End(\mf h)$ be the 
		adjoint actions of $G$ and $H$, respectively, on their Lie algebras. 
		Then the composite map 
		$$\beta\circ f : G \to \End(\mf h)$$ 
		gives an adjoint action of $G$ on $\mf h$. 
		Since $df$ is a $G$--module homomorphism and $G$ is reductive, 
		there is a $G$--submodule $W$ of $\mf h$ such that 
		\begin{equation}\label{eqn:direct-sum-decomp}
			\mf h = df(\mf g)\bigoplus W 
		\end{equation}
		as $G$--modules. Since $df$ is injective, from the direct sum decomposition of 
		$G$--modules in \eqref{eqn:direct-sum-decomp} projecting to the first factor 
		we get a $G$--module homomorphism $\pi_{\mf g} : \mf h \to \mf g$ such that 
		$\pi_{\mf g}\circ df = \Id_{\mf g}$. Then $\pi_{\mf g}$ induces a vector bundle 
		homomorphism 
		\begin{equation}
			\widetilde{\pi_{\mf g}} : \ad(E_H) \longrightarrow \ad(E_G) 
		\end{equation}
		such that $\widetilde{\pi_{\mf g}}\circ\ad(f) = \Id_{{\rm ad}(E_G)}$. 
		
		Suppose that $E_H$ admits a $\mc V$--valued Lie algebroid connection. 
		Then there exists a $\mc O_X$--module homomorphism 
		$$\eta : \At_{\varphi}(E_H) \to \ad(E_H)$$ 
		such that $\eta\circ\iota_H = \Id_{\ad(E_H)}$, 
		where $\iota_H : \ad(E_H) \to \At_{\varphi}(E_H)$ is the homomorphism 
		in \eqref{diag:comm-diag-of-Atiyah-exact-seqn-for-extn-of-str-grp}. 
		Then it follows from the commutativity of the diagram 
		\eqref{diag:comm-diag-of-Atiyah-exact-seqn-for-extn-of-str-grp} 
		that the composition 
		$$\widetilde{\pi_{\mf g}}\circ\eta\circ\At(f) : \At_{\varphi}(E_G) \to \ad(E_G)$$
		gives an $\mc O_X$--linear splitting of the top exact sequence in 
		\eqref{diag:comm-diag-of-Atiyah-exact-seqn-for-extn-of-str-grp}. 
		Thus $E_G$ admits a $\mc V$--valued Lie algebroid connection. 
	\end{proof}
	
	Now we consider the case when the structure group of a principal bundle is not reductive. 
	Let $G$ be a reductive linear algebraic group over $\bb C$. 
	A closed subgroup $P$ of $G$ is said to be {\it parabolic} if $G/P$ is a complete $\bb C$--variety. 
	Let $P$ be a parbolic subgroup of $G$.  
	Let $\mf{R}_u(P)$ be the unipotent radical of $P$, and let 
	$$q : P \longrightarrow P/\mf{R}_u(P)$$ 
	be the associated quotient map. Let $L \subseteq P$ be a {\it Levi factor of $P$}; 
	a closed connected subgroup of $P$ such that $q\big\vert_L : L \to P/\mf{R}_u(P)$ 
	is an isomorphism of algebraic groups over $\bb C$. Note that $L$ is reductive. 
	Given a principal $P$--bundle $E_P$ on $X$, let $E_L := E_P\times^{q'}L$ be the 
	principal $L$--bundle on $X$ obtained by extending the structure group of $E_P$ 
	along the homomorphism $$q' := (q\big\vert_L)^{-1}\circ\,q : P \to L.$$ 
	The action of $P$ on the nilpotent radical $\mf{n} := \Lie(\mf{R}_u(P))$ of the 
	Lie algebra $\mf{p} := \Lie(P)$ gives rise to a subbundle 
	$E_P(\mf{n}) := E_P\times^P\mf{n}$ of the adjoint bundle $\ad(E_P)$ of $E_P$, 
	and the associated quotient bundle 
	$$\ad(E_P)/E_P(\mf{n}) \cong E_P(\mf{l}) = \ad(E_L),$$ 
	where $\mf{l} = \Lie(L)$ is the Lie algebra of $L$. 
	Then we have the following commutative diagram of vector bundle homomorphisms with all 
	rows and columns exact 
	(c.f. \eqref{diag:phi-twisted-diag-of-Atiyah-exact-seqns-for-extn-of-str-grp}): 
	\begin{equation}
		\begin{gathered}
			\xymatrix{
				& 0 \ar[d] & 0 \ar[d] & & & \\ 
				& E_P(\mf{n}) \ar@{=}[r] \ar[d] & E_P(\mf{n}) \ar[d] & & & \\ 
				0 \ar[r] & \ad(E_P) \ar[r] \ar[d]^-{\ad(q')} & \At_{\varphi}(E_P) \ar[d]^-{\At(q')} 
				\ar[r]^-{\widetilde{\rho}_P} & V \ar@{=}[d] \ar[r] & 0 \\ 
				0 \ar[r] & \ad(E_L) \ar[d] \ar[r] & \At_{\varphi}(E_L) \ar[d] 
				\ar[r]^-{\widetilde{\rho}_L} & V \ar[r] & 0 \\ 
				& 0 & 0 & &&& 
			}
		\end{gathered}
	\end{equation}
	Suppose that $E_L$ admits a $\mc V$--valued Lie algebroid connection 
	$\nabla : V \to \At_{\varphi}(E_L)$. Then $\widetilde{\rho}_L\circ\nabla = \Id_V$. 
	Then the subsheaf 
	$$\mc{E}_\nabla := \At(q')^{-1}\left(\nabla(V)\right) \subseteq \At_{\varphi}(E_P)$$ 
	fits into the following short exact sequence of $\mc{O}_X$--modules 
	\begin{equation}\label{eqn:short-ext-seq-for-reduction-of-str-grp-connection}
		0 \to E_P(\mf{n}) \to \mc{E}_\nabla \to V \to 0 
	\end{equation}
	on $X$ whose splitting gives rise to a $\mc V$--valued connection on $E_P$. 
	Note that the above short exact sequence 
	\eqref{eqn:short-ext-seq-for-reduction-of-str-grp-connection} defines a cohomology class 
	\begin{equation}
		\Phi(E_P, L, \nabla) \in H^1(X, E_P(\mf{n})\otimes V^*), 
	\end{equation}
	which vanishes if and only if the exact sequence in 
	\eqref{eqn:short-ext-seq-for-reduction-of-str-grp-connection} splits $\mc{O}_X$--linearly. 
	From this, we have the following result. 
	
	\begin{proposition}\label{prop:connection-on-E_P-from-E_L}
		With the above notations, if $H^1(X, E_P(\mf{n})\otimes V^*) = 0$, then a 
		$\mc V$--valued Lie algebroid connection on $E_L$ gives rise to a 
		$\mc V$--valued Lie algebroid connection on $E_P$. 
	\end{proposition}

	\section{Existence of Lie Algebroid Connections}\label{sec:Existence}
	In this section we prove existence of Lie algebroid connection on a principal $G$--bundle $E_G$ on a smooth complex projective curve $X$. We first prove it for the case of semistable principal bundles, and then prove the general case by analyzing the Harder-Narasimhan reduction of $E_G$.

	\subsection{Semistable principal  $G$--bundles}\label{sec:semistability}
	
	Let $X$ be an irreducible smooth complex projective variety of dimension $n \geq 1$. Fix an ample line bundle $H$ on $X$. Let $E$ be a torsion-free coherent sheaf of $\mc O_X$--modules $E$ on $X$. Note that $E$ is locally free outside a closed subset of codimension $\geq 2$. The {\it degree} of $E$ is defined by 
	$$\deg(E) := \int_X c_1(E)\cup [H]^{n-1} \in \bb Z,$$
	where $c_1(E)$ stands for the first Chern class of the determinant line bundle $\det(E)$ of $E$. Note that $\deg(E)$ depends on choice of an ample class $[H]$ on $X$. 
	The rational number 
	$$\mu(E) := \frac{\deg(E)}{{\rm rank}(E)}$$ 
	is called the {\it slope} of $E$. 
	
	\begin{definition}
		A torsion-free coherent sheaf $E$ on $X$ is said to be {\it stable} (resp., {\it semistable}) if for any non-zero proper subsheaf $F$ of $E$ we have $\mu(F) < \mu(E)$ (resp., $\mu(F) \leq \mu(E)$). 
	\end{definition}
	
	The notion of slope semistablity and stability has a natural generalization to the case of principal $G$--bundles due to Ramanathan \cite{Ramanathan-1975, Ramanathan-I, Anchouche-Biswas-2001}, which coincides with the corresponding notion for rank $r$ vector bundles by taking $G = \GL_r(\bb C)$. 
	Let $G$ be a reductive linear algebraic group over $\bb C$. Let 
	$$p : E_G \longrightarrow X$$ be a principal $G$--bundle on $X$. 
	Recall that a reduction of structure group of $E_G$ to a closed subgroup $P \subseteq G$ over an open subset $U \subseteq X$ is given by a section 
	$$\sigma \in \Gamma(U, E_G/P)$$
	over $U$ of the fiber bundle $E_G/P \longrightarrow X$. If $\pi : E_G \to E_G/P$ is the projection map, then for given a section $\sigma : U \longrightarrow E_G/P$ of the fiber bundle $E_G/P \to X$, the inverse image $\pi^{-1}(\sigma(U)) \subseteq E_G$ defines a principal $P$--bundle $E_P$ over $U$ such that the principal $G$--bundle $E_P\times^P G$ on $U$ obtained by extension of structure group of $E_P$ along the inclusion homomorphism $P \hookrightarrow G$ is isomorphic to ${E_G}\big\vert_U$. 
	
	\begin{definition}
		A principal $G$--bundle $p : E_G \to X$ is said to be {\it stable} (resp., {\it semistable}) if for any maximal proper parabolic subgroup $P \subset G$ and for every reduction $\sigma : U \to E_G/P$ of its structure group to an open subset $U \subset X$ with ${\rm codim}_X(X\setminus U) \geq 2$, we have 
		$$\deg\left(\sigma^*(T_{\rm rel})\right) > 0 \ (\text{resp., } \,\geq\, 0),$$ 
		where $T_{\rm rel}$ is the relative tangent bundle for the natural projection $E_G/P \longrightarrow X$. 
	\end{definition}
	
	Semistability of principal $G$--bundle $E_G$, with $G$ a reductive affine algebraic group over $\bb C$, can be characterized in terms of semistability of its adjoint vector bundle.

	\begin{lemma}
		\cite{Anchouche-Biswas-2001}\label{lem:E_G-semistable-iff-ad(E_G)-semistable} 
		A principal $G$--bundle $E_G$ on $X$ is semistable if and only if its adjoint vector bundle $\ad(E_G) := E_G\times^G\mf{g}$ on $X$ is semistable. 
	\end{lemma}

	\subsection{Existence for semistable $G$--bundles}\label{sec:semistable-case-existence}
	Let $X$ be an irreducible smooth projective curve over $\bb C$. 
	Fix a Lie algebroid $\mc V = (V, [\cdot, \cdot], \varphi)$ on $X$ 
	such that the underlying vector bundle $V$ of $\mc V$ is stable. 
	Let 
	$$\mu(\mc V) := \frac{\deg(V)}{{\rm rank}(V)}$$ 
	be the {\it slope} of the underlying vector bundle $V$ of the Lie algebroid $\mc V$. Note that $TX$ is a line bundle on $X$ with the slope $\mu(TX) = 2-2g$, where $g$ is the genus of $X$. 
	
	If $\mu(\mc V) > 2-2g = \mu(TX)$, then both $V$ and $TX$ being stable vector bundles 
	on $X$ we have $H^0(X, \Hom(V, TX)) = 0$ (see \cite[Proposition 1.2.7]{Huybrechts-Lehn-2010}), 
	and hence $\varphi = 0$ in this case. Then a $\mc V$--valued Lie algebroid connection 
	on $E_G$ is just a global section of $\ad(E_G)\otimes V^*$; so we may take the 
	zero section in $H^0(X, \ad(E_G)\otimes V^*)$, in particular. 
	
	If $\mu(\mc V) = 2-2g = \mu(TX)$, then any non-zero $\mc O_X$--module homomorphism 
	$\varphi : V \to TX$ is an isomorphism (see \cite[Proposition 1.2.7]{Huybrechts-Lehn-2010}). 
	Then we may replace $V$ with $TX$ so that a $\mc V$--valued Lie algebroid connection 
	on $E_G$ is nothing but a holomorphic connection on $E_G$. This case is studied in 
	detail in \cite{Azad-Biswas-2002}. 
	
	Now we assume that $\mu(\mc V) < 2-2g = \mu(TX)$. Then we have the following. 
	
	\begin{proposition}\label{prop:existence_semistable-G-bundles-on-curve}
		Let $G$ be a reductive linear algebraic group over $\bb C$. 
		With the above assumptions on $\mc V$, any semistable principal $G$--bundle 
		$E_G$ on $X$ admits a $\mc V$--valued Lie algebroid connection. 
	\end{proposition}
	
	\begin{proof}
		Let $\Phi_{\mc V}(E_G) \in H^1(X, \ad(E_G)\otimes V^*)$ be the $\mc V$--valued Atiyah class 
		of $E_G$. By Serre duality, we have 
		\begin{equation*}
			H^1(X, \ad(E_G)\otimes V^*) \cong H^0(X, \ad(E_G)^*\otimes V\otimes K_X)^*,
		\end{equation*} 
		where $K_X = \Omega_X^1$ is the canonical line bundle on $X$. 
		Since $E_G$ is semistable by assumption, its adjoint bundle $\ad(E_G)$ is semistable 
		by \cite[Proposition 2.10]{Anchouche-Biswas-2001}. 
		Then the tensor product bundle $\ad(E_G)^*\otimes V\otimes K_X$ is semistable 
		(see \cite[Theorem 3.1.4]{Huybrechts-Lehn-2010}). 
		Since $G$ is reductive, there is a $G$-invariant non-degenerate symmetric $\bb C$--bilinear form $\kappa$ on its Lie algebra $\mf g = \mf z(\mf g)\oplus[\mf g, \mf g]$, where $\mf z(\mf g)$ is the center of $\mf g$; c.f. \cite{Springer-98}. This gives an isomorphism of the adjoint vector bundle $\ad(E_G)$ with its dual vector bundle $\ad(E_G)^*$ on $X$. From this we have $\deg(\ad(E_G)) = 0$, and hence 
		$$\mu(\ad(E_G)^*\otimes V\otimes K_X) = \mu(K_X)+\mu(V) = 2g-2+\mu(V)<0.$$
		Then by \cite[Proposition 1.2.7]{Huybrechts-Lehn-2010} $H^0(X, \ad(E_G)^*\otimes V\otimes K_X) = 0$, 
		and hence $\Phi_{\mc V}(E_G) = 0$. Hence the result follows. 
	\end{proof}

	\subsection{The case of non-semistable $G$--bundles}\label{sec:main-thm}
	Now consider the case when $E_G$ is not semistable. 
	We study this case by analyzing Harder-Narasimhan filtration of the adjoint vector bundle of $E_G$. To make the exposition self-contained, we briefly recall some results from \cite{Anchouche-Biswas-2001}.

	Let $G$ be a reductive affine algebraic group over $\bb C$. Then there is a $G$--invariant non-degenerate symmetric bilinear form on its Lie algebra $\mf{g} = \mf{z}(\mf g)\oplus[\mf g, \mf g]$, where $\mf z(\mf g)$ is the center of $\mf g$. Fix such a form $\kappa$ on $\mf g$ and call it the ``Killing form'' on $\mf g$. Let $E_G$ be a principal $G$--bundle on $X$. Suppose that $E_G$ is not semistable, then $E := \ad(E_G)$ is not semistable by Lemma \ref{lem:E_G-semistable-iff-ad(E_G)-semistable}. 
	Let 
	\begin{equation}\label{eqn:HN-filtration-of-ad(E_G)}
		0 = E_0 \subset E_1 \subset \cdots \subset E_k = \ad(E_G)
	\end{equation}
	be the Harder-Narasimhan filtration of $E := \ad(E_G)$ by its subsheaves $E_i$, with $E/E_i$ torsion free, for all $i = 1, \ldots, k$. Then there exists a non-empty open subset $U \subseteq X$ with ${\rm codim}_X(X\setminus U) \geq 2$ such that $E_i$ and $E/E_i$ are locally free on $U$, for all $i \geq 0$. Let $E_j^{\perp} \subseteq E$ be the subsheaf of $E = \ad(E_G)$ whose fiber at each $x \in U$ is 
	$$E_{j, x}^\perp := \left\{v \in E_x : \kappa(v, E_{j, x}) = 0\right\}.$$
	Since $\kappa$ is non-degenerate, $E_j^\perp\big\vert_U \cong {(E/E_j)^*}\big\vert_U$, for all $i \geq 0$. Moreover, $E^* \cong E$. Since the dual $F^*$ of a semistable torsion-free coherent sheaf $F$ on $X$ is again semistable with slope $\mu(F^*) = -\mu(F)$ and the the filtration 
	\begin{equation}\label{eqn:HN-filtration-perp}
		0 = E_k^\perp \subset E_{k-1}^\perp \subset \cdots \subset E_1^\perp \subset E_0^\perp 
	\end{equation} 
	satisfies $\mu(E_{k-1}^\perp) > \mu(E_{k-2}^\perp/E_{k-1}^\perp) > \cdots > \mu(E_{0}^\perp/E_{1}^\perp)$, by uniqueness of Harder-Narasimhan filtration (c.f. \cite[Theorem 1.3.4]{Huybrechts-Lehn-2010}) of $E$ the filtration \eqref{eqn:HN-filtration-perp} coincides with the filtration  \eqref{eqn:HN-filtration-of-ad(E_G)}. Therefore, $E_j^\perp = E_{k-j}$, for all $j = 0, 1, \ldots, k$. 
	Therefore, by reindexing, we can rewrite the filtration \eqref{eqn:HN-filtration-of-ad(E_G)} as 
	\begin{equation*}
		0 = F_{-(\ell+1)} \subset F_{-\ell} \subset \cdots \subset F_{-1} \subset F_0 \subset F_1 \subset \cdots F_{\ell-1} \subset F_{\ell} = E, 
	\end{equation*}
	where $F_{-i}$ is orthogonal to $F_{i-1}$ with respect to the Killing form $\kappa$ on the Lie algebra bundle $E = \ad(E_G)$.  
	It follows from the slope inequality and semistability of the successive quotients of terms of the Harder-Narasimhan filtration of $E$ that the following composite maps 
	\begin{align*}
		& F_0 \otimes F_0 \stackrel{[\cdot\,,\,\cdot]}{\longrightarrow} E \twoheadrightarrow E/F_0, \\ 
		\text{and } \ \ 
		& F_{-i}\otimes F_{-i-1} \stackrel{[\cdot\,,\,\cdot]}{\longrightarrow} E \twoheadrightarrow E/F_{-i-1} 
	\end{align*}
	vanishes for all $i \geq 0$. Therefore, $F_0$ is a sheaf of Lie subalgebras of $E = \ad(E_G)$ and 
	$$[F_{-i}, F_{-1}] \subseteq F_{-i-1}, \ \forall\ i \geq 0.$$ 
	In particular, $F_{-1}$ is a sheaf of nilpotent Lie subalgebra of $F_0$, which is also an ideal of $F_0$. Moreover, by \cite[Lemma II.2]{Faltings-93}, $F_0$ is a sheaf of parabolic Lie subalgebra of $E = \ad(E_G)$ which, by \cite[Lemma 2.11]{Anchouche-Biswas-2001}, gives a reduction $E_P \subset E_G$ of structure group of $E_G$ to a parabolic subgroup $P \subset G$ outside a closed subset of $X$ of codimension $\geq 2$, known as the Harder-Narasimhan reduction or the canonical reduction of $E_G$. In particular, $\ad(E_P) = F_0$. 
	
	Let $\mf{R}_u(P)$ be the unipotent radical of the parabolic subgroup $P$ of $G$. Fix a maximal torus $T$ of $G$. Then the quotient $P/\mf{R}_u(P)$ is isomorphic to a closed connected reductive subgroup $L$ of $P$ such that the restriction of the quotient map $q : P \to P/{\mf R}_u(P)$ to $L \subset P$ is an isomorphism of algebraic groups. Then the principal $L$--bundle 
	$$E_L := E_P\times^q L$$ 
	on $X$ obtained by extension of the structure group of $E_P$ by the quotient homomorphism 
	$$q : P \to P/\mf{R}_u(P) \cong L$$ 
	is semistable (see \cite[Theorem 1 and Lemma 6]{HN-reduction-Anchouche-Azad-Biswas-2002}). 
	Then by Proposition \ref{prop:existence_semistable-G-bundles-on-curve} the principal  $L$--bundle $E_L$ admits a $\mc V$--valued Lie algebroid connection. Now the action of $P$ on the nipotent radical $\mf n := {\rm Lie}(\mf{R}_u(P))$ of the Lie algebra of $P$ gives rise to a subbundle $E_P(\mf n) = E_P\times^P\mf n$ of the adjoint bundle $\ad(E_P)$ of $E_P$. Note that $E_P(\mf n) \cong F_{-1}$ by  \cite[Lemma 6]{HN-reduction-Anchouche-Azad-Biswas-2002}, 
	where $F_{-1}$ is the sheaf of nilpotent Lie subalgebra of $F_0 = \ad(E_P)$ as mentioned in the preceding paragraph. Since for any surjective $\mc O_X$-module homomorphism of torsion-free coherent sheaves $\mc E \to \mc E'$ we have $\mu_{\min}(\mc E) \geq \mu_{\min}(\mc E')$ (see e.g., \cite[p.~29]{Huybrechts-Lehn-2010}), it follows from \cite[Lemma 7]{HN-reduction-Anchouche-Azad-Biswas-2002} that 
	$$\mu_{\min}(E_P(\mf n)) > 0.$$ 
	Since $V\otimes K_X$ is semistable with $\mu(V\otimes K_X) < 0$, 
	it follows from \cite[Proposition 1]{HN-reduction-Anchouche-Azad-Biswas-2002} 
	that $\Hom(E_P(\mf n), V\otimes K_X) = 0$, and hence 
	by Serre duality we have $$H^1(X, E_P(\mf n)\otimes V^*) = 0.$$ 
	Then by Proposition \ref{prop:connection-on-E_P-from-E_L} 
	the principal $P$--bundle $E_P$ admits a $\mc V$--valued Lie algebroid connection, and then 
	by Proposition \ref{prop:Lie-algebroid-connection-under-extn-of-str-grp} the principal 
	$G$--bundle $E_G$ admits a $\mc V$--valued Lie algebroid connection. 
	Thus we obtain the following.

	\begin{theorem}\label{thm:main-thm}
		Fix a Lie algebroid $\mc V = (V, [\cdot\,,\,\cdot], \varphi)$ on $X$ such that $V$ is stable with $\mu(V) < 2-2g = \deg(TX)$. 
		Let $G$ be a reductive linear algebraic group over $\bb C$. 
		Let $E_G$ be a principal $G$--bundle on $X$. 
		Then $E_G$ admits a $\mc V$--valued Lie algebroid connection. 
	\end{theorem}

	\section*{Acknowledgment}
	The authors would like to thank the anonymous referee for carefully reading the
	manuscript and providing useful comments and suggestions to improve it. 
	The first named author is supported by the \textit{National Board of Higher Mathematics (NBHM)} 
	through the Doctoral Research Fellowship Program. 
	The second named author is partially supported by the DST INSPIRE Faculty Fellowship 
	(Research Grant No.: DST / INSPIRE / 04 / 2020 / 000649, IFA20-MA-144), the Ministry of Science \& Technology, 
	Government of India. 
	

\end{document}